\documentclass{amsart}

\usepackage{amssymb,amsmath,amsthm}
\usepackage{amsrefs}

\newcommand{\Einfty}{E(F_2,2)}
\newcommand{\Pw}{\mathbb{P}_{\omega,2^{< \omega}}}
\newcommand{\Pt}{\mathbb{P}_{\omega,2}}
\newcommand{\halts}{\!\!\downarrow} 

\newcommand{\embeds}{\sqsubseteq}

\newcommand{\restrict}{\restriction}
\newcommand{\cantor}{2^\omega}

\newcommand{\define}[1]{\emph{#1}}

\newcommand{\join}{\oplus}
\newcommand{\bigjoin}{\bigoplus}

\newcommand{\bigunion}{\bigcup}
\newcommand{\disjointunion}{\sqcup}

\DeclareMathOperator{\dom}{dom}
\DeclareMathOperator{\ran}{ran}

\newcommand{\AD}{\mathrm{AD}}
\newcommand{\ZF}{\mathrm{ZF}}
\newcommand{\ZFC}{\mathrm{ZFC}}
\newcommand{\DC}{\mathrm{DC}}
\newcommand{\V}{\mathrm{V}}
\newcommand{\gL}{\mathrm{L}} 
\newcommand{\J}{\mathrm{J}}

\newcommand{\kO}{\mathcal{O}} 
\newcommand{\Z}{\mathbb{Z}}

\newcommand{\DDelta}{\mathbf{\Delta}}
\newcommand{\PPi}{\mathbf{\Pi}}
\newcommand{\SSigma}{\mathbf{\Sigma}}

\theoremstyle{plain}
\newtheorem{thm}{Theorem}[section]
\newtheorem{lemma}[thm]{Lemma}
\newtheorem{prop}[thm]{Proposition}
\newtheorem*{prop*}{Proposition}

\newtheorem{question}[thm]{Question}

\newtheorem{conj}[thm]{Conjecture}

\theoremstyle{definition}
\newtheorem{defn}[thm]{Definition}

\begin{document}

\title[Martin's conjecture, $\equiv_A$, and countable Borel equivalence
relations]{Martin's conjecture, arithmetic equivalence, and countable Borel equivalence relations}

\author{Andrew Marks}
\address{Department of Mathematics, California Institute of Technology}
\email{marks@caltech.edu}

\author{Theodore Slaman}
\address{Department of Mathematics, University of California at Berkeley}
\email{slaman@math.berkeley.edu}

\author{John Steel}
\address{Department of Mathematics, University of California at Berkeley}
\email{steel@math.berkeley.edu}

\thanks{
The results in Section 3 form part of the PhD thesis of the
first author, which was done under the supervision of Ted Slaman.
The first author would also like to thank Ben Miller for helpful
conversations.
The research of the second author was partially supported by NSF award
DMS-1001551. The research of the third author was partially supported by
NSF award DMS-0855692.
}

\begin{abstract}There is a fascinating interplay and overlap between recursion
theory and descriptive set theory. A particularly beautiful source of such
interaction has been Martin's conjecture on Turing invariant functions.
This longstanding open problem in recursion theory has connected to many
problems in descriptive set theory, particularly in the theory of countable
Borel equivalence relations.

In this paper, we shall give an overview of some work that has been done on
Martin's conjecture, and applications that it has had in descriptive set
theory. We will present a long unpublished result of Slaman and Steel that
arithmetic equivalence is a universal countable Borel equivalence relation.
This theorem has interesting corollaries for the theory of universal
countable Borel equivalence relations in general. We end with some open
problems, and directions for future research.
\end{abstract}

\maketitle

\section{Introduction}
\subsection{Martin's conjecture}
\label{sec:MC}

Martin's conjecture on Turing invariant functions is one of the
oldest and deepest open problems on the global structure of the Turing
degrees. Inspired by Sacks' question on the existence of a degree-invariant
solution to Post's problem \cite{MR0186554}, Martin made a sweeping
conjecture that says in essence, the only nontrivial definable
Turing invariant functions are the Turing jump and its iterates through
the transfinite. 

Our basic references for descriptive set theory and effective descriptive
set theory are the books of Kechris \cite{MR1321597} and
Sacks \cite{MR1080970}. Let $\leq_T$ be Turing reducibility on the Cantor
space $\cantor$, and let $\equiv_T$ be Turing equivalence. Given $x \in
\cantor$, let $x'$ be the Turing jump of $x$. The Turing degree of a real
$x \in \cantor$ is the $\equiv_T$ equivalence class of $x$. A \define{Turing invariant
function} is a function $f: \cantor \to \cantor$ such that for all reals
$x, y \in \cantor$, if $x \equiv_T y$, then $f(x) \equiv_T f(y)$. The
Turing invariant functions are those which induce functions on the Turing
degrees.

With the axiom of choice, we can construct many pathological Turing
invariant functions. Martin's conjecture is set in the context of
$\ZF + \DC + \AD$, where $\AD$ is the axiom of determinacy. We assume $\ZF
+ \DC + \AD$ for the rest of this section. The results we
will discuss all ``localize'' so that the assumption of $\AD$ essentially
amounts to studying definable functions assuming definable determinacy,
for instance, Borel functions using Borel determinacy.

To state Martin's conjecture, we need to recall the notion of Martin
measure. A \define{Turing cone} is a set of the form $\{x : x \geq_T y\}$.
The real $y$ is said to be the \define{base} of the cone $\{x : x \geq_T y\}$. A
\define{Turing invariant set} is a set $A \subseteq \cantor$ that is closed under
Turing equivalence: for all $x, y \in \cantor$, if $x \in A$ and $x
\equiv_T y$, then $y \in A$.
Martin has shown that under $\AD$, every Turing invariant set $A$ either
contains a Turing cone, or is disjoint from a Turing cone~\cite{MR0227022}.
Note that
the intersection of countably many cones contains a cone; the intersection
of the cones with bases $\{y_i\}_{i \in \omega}$ contains the cone whose
base is the join of the $y_i$.
Hence, under $\AD$, the function 
\[ \mu(A) = \begin{cases} 1 & \text{ if $A$ contains a Turing cone}\\
0 & \text{ if the complement of $A$ contains a Turing cone}
\end{cases}
\]
is a measure on the $\sigma$-algebra of Turing invariant sets. 
This measure is called Martin measure. For the rest of this section, by
a.e.\ we will mean almost everywhere with respect to Martin measure. Since
we will care only about the behavior of functions a.e., we will occasionally
deal with functions which are only defined a.e.

For Turing invariant $f, g : \cantor \to \cantor$, let $f \leq_m g$ if and
only if
$f(x) \leq_T g(x)$ a.e. Similarly, $f \equiv_m g$ if and only if
$f(x) \equiv_T g(x)$ a.e. Say that $f$ is increasing a.e.\ if $f(x) \geq_T
x$ a.e. Finally, say that $f$ is constant a.e.\ if 
there exists a $y \in \cantor$ such that $f(x) \equiv_T y$ a.e.\ (i.e.
the induced function on Turing degrees is constant a.e.).

We are now ready to state Martin's conjecture on Turing invariant functions.
\begin{conj}[Martin~{\cite[p.~281]{MR526912}}] \label{conj:MC}
Assume $\ZF + \DC + \AD$. Then 
\begin{enumerate}
\item[I.] If $f: \cantor \to \cantor$ is Turing invariant, and $f$ is not
increasing a.e.\ then $f$ is constant a.e.
\item[II.] $\leq_m$ prewellorders the set of Turing invariant functions
which are increasing a.e. If $f$ has $\leq_m$-rank $\alpha$, then $f'$ has
$\leq_m$-rank $\alpha+1$, where $f'(x) = f(x)'$ for all $x$.
\end{enumerate}
\end{conj}

While Martin's conjecture remains open, significant progress has been made
towards establishing its truth. Let $\varphi_i$ be the $i$th partial
recursive function. 
Say that $x \geq_T y$ via $i$ if $\varphi_i(x) = y$. Say that $x \equiv_T y$
via $(i,j)$ if $x \geq_T y$ via $i$ and $y \geq_T x$ via $j$.
Suppose that $f$
is a Turing invariant function.
Say that $f$ is
\define{uniformly Turing invariant} if there exists a function $u: \omega^2
\to \omega^2$ so that if $x \equiv_T y$ via $(i,j)$, then $f(x) \equiv_T
f(y)$ via $u(i,j)$. Note that our definition of a uniformly Turing
invariant function is slightly different that the definitions of some of
the papers we reference\footnote{
In particular, the definition of uniformly Turing invariant that we give here is different
than the definitions used in \cite{MR960895} and \cite{MR654792} (which
also differ from each other). The definition in \cite{MR960895} requires
only that there is a pointed perfect tree $T$ and a function $u: \omega^2
\to \omega^2$ such that for all $x,y \in [T]$, if $x \equiv_T y$ via
$(i,j)$, then $f(x) \equiv_T f(y)$ via $u(i,j)$. This definition is related
to our definition in the following way:
\begin{prop*}
  Let $f$ be Turing invariant. Then there exists a uniformly Turing
  invariant $g$ (using our definition) which is defined a.e.\ such that 
  $g \equiv_m f$ if and only if 
  there is a pointed perfect tree $T$ and
  a function $u: \omega^2 \to \omega^2$ such that for all $x,y \in [T]$,
  if $x \equiv_T y$ via $(i,j)$, then $f(x) \equiv_T f(y)$ via $u(i,j)$.
\end{prop*}
\begin{proof}
  For the forward direction, consider the function $\pi$ which maps $x$ to
  the lexicographically least $(i,j)$ such that $f(x) \equiv_T g(x)$ via
  $(i,j)$. Now use Lemma~\ref{lemma:pps} to find a pointed
  perfect set on which $\pi$ is constant. 

  For the reverse direction, first let $[T]$ be the pointed perfect set on
  which $f$ is uniform. Pass to a uniformly pointed perfect tree
  $\hat{T} \subseteq T$ where there exists an $e$ such that for all $x \in
  [\hat{T}]$, we have $\varphi_e(x) = \hat{T}$. Now define the uniformly
  Turing invariant $g$ by composing $f$ with the canonical homeomorphism
  from $\cantor$ to $[\hat{T}]$, and then restricting to the cone $\{x : x
  \geq_T \hat{T}\}$.
\end{proof}

Similar techniques along with Theorem~\ref{thm:Slaman-Steel} can be used to
show that that for all Turing invariant $f$, there exists a $g \equiv_m f$
that is uniformly Turing invariant in our sense if and only if
there exists an $h \equiv_m f$ that is uniformly Turing
invariant in the sense of \cite{MR654792} (where $g$ and $h$ are both
defined a.e.).

Thus, the differences between these definitions are harmless; we have not
changed what it means for a function on Turing degrees to be represented
a.e.\ by a uniformly invariant function. We use our definition because of
its simplicity, and because it generalizes in
Section~\ref{sec:conjectures} more readily than the definitions of
\cite{MR654792} and \cite{MR960895}.
}.

The first progress on Martin's conjecture was made by 
Steel \cite{MR654792} and was continued by Slaman and
Steel \cite{MR960895}. They proved that Martin's conjecture is true when
restricted to the class of uniformly Turing invariant functions. 

\begin{thm}[Slaman and Steel \cite{MR960895}]\label{thm:Slaman-Steel}
Part I of Martin's conjecture holds for all uniformly Turing
invariant functions.
\end{thm}
\begin{thm}[Steel \cite{MR654792}]\label{thm:Steel}
Part II of Martin's conjecture holds 
for all uniformly Turing
invariant functions.
\end{thm}

Theorems~\ref{thm:Slaman-Steel} and \ref{thm:Steel} also imply that
Martin's conjecture is true when restricted to the larger class of
functions $f$ so that $f \equiv_m g$ for some
uniformly Turing invariant $g$. Steel has conjectured that this is true of
all Turing invariant functions.

\begin{conj}[Steel \cite{MR654792}]\label{conj:SC}
  If $f: \cantor \to \cantor$ is Turing invariant, then there exists a
  uniformly Turing invariant $g$ which is defined a.e., and $f \equiv_m g$. 
\end{conj}

Assuming Conjecture~\ref{conj:SC}, Steel~\cite{MR654792} has computed the
$\leq_m$-rank of many familiar such jump operators. Steel also proves that
Conjecture~\ref{conj:SC} implies that if $f(x) \in L[x]$ a.e., then $f$ has
a natural normal form in terms of master codes in Jensen's $\J$. 

The original intent of Martin's conjecture was to be a precise way of
stating that the only definable non-constant Turing invariant functions are
the Turing jump and its transfinite iterates such as $x \mapsto
x^{(\alpha)}$ for $\alpha < \omega_1$, $x \mapsto \kO^x$, and $x \mapsto
x^\sharp$. Becker has shown that Conjecture~\ref{conj:SC} precisely captures this
idea. In \cite{MR960994}, Becker defines the notion of a reasonable
pointclass, and shows that for any such reasonable pointclass $\Gamma$, for
every $x$, there is a universal $\Gamma(x)$ set, where $\Gamma(x)$ is the
relativization of $\Gamma$ to $x$. (Such a universal set is not unique, but
the universal $\Gamma(x)$ subset of $\omega$ will be unique up to Turing
equivalence.) For instance, if we consider the pointclass of $\Pi^1_1$
sets, the universal $\Pi^1_1(x)$ subset of $\omega$ is $\kO^x$. Becker has shown
that the strictly increasing uniformly Turing invariant functions are
precisely the functions which map $x$ to the universal $\Gamma(x)$ subset
of $\omega$ for some reasonable pointclass $\Gamma$.

\begin{thm}[Becker \cite{MR960994}]
  Let $f:\cantor \to \cantor$ be a Turing invariant function so that $f(x)
  >_T x$ a.e. Then $f$ is uniformly Turing invariant if and only if there is a
  reasonable pointclass $\Gamma$, and a Turing invariant $g$ so that $g(x)$
  is the universal $\Gamma(x)$ subset of $\omega$, and $f \equiv_m g$. 
\end{thm}

Suppose $f$ is Turing invariant. Then say that $f$ is \define{order
preserving} if $x \geq_T y$ implies that $f(x) \geq_T f(y)$. Say that $f$
is \define{uniformly order preserving} if there exists a function $u: \omega \to
\omega$ so that $x \geq_T y$ via i implies $f(x) \geq_T y$ via $u(i)$. It is
clear that if $f$ is uniformly order preserving then $f$ is uniformly Turing
invariant. A corollary of Becker's work is that for any Turing invariant
$f$, there exists a uniformly Turing invariant $g$ so that $g \equiv_m f$
if and only if there exists a uniformly order preserving $h$ so that $h
\equiv_m f$.

Two more cases of Martin's conjecture are known. They are especially
interesting because they do not require 
uniformity assumptions. 
\begin{thm}[Slaman and
Steel \cite{MR960895}]\label{thm:inc_order_preserving}
  If $f$ is a Borel order preserving Turing invariant function that is
  increasing a.e., then there exists an $\alpha < \omega_1$ so that $f(x)
  \equiv_T x^{(\alpha)}$ a.e.  
\end{thm}
The proof of this theorem uses a generalization of the Posner-Robinson theorem for
iterates of the Turing jump up through $\omega_1$. To
generalize this theorem beyond the Borel functions, it would
be enough to generalize the Posner-Robinson theorem further through the
hierarchy of jump operators. For instance, Woodin \cite{MR2449474} has proved a
generalization of the Posner-Robinson theorem for the hyperjump.
This can be used to show that if $f$ is
increasing and order preserving a.e., and not Borel, then $f(x) \geq_T \kO^x$ a.e.

The last known case of Martin's conjecture is for all
recursive functions. 
\begin{thm}[Slaman and Steel \cite{MR960895}]\label{thm:continous_MC}
  Suppose $f(x) \leq_T x$ a.e. Then either $f(x) \equiv_T x$ a.e., or $f$
  is constant a.e.
\end{thm}
The proof of this theorem uses both game arguments and a significant amount
of recursion theory. Generalizing this theorem past the $\DDelta^0_1$
functions appears to be difficult, and the proof does not give much of an
idea of how to do this.

The special case of a degree invariant solution to Post's problem has also
received considerable attention. Lachlan \cite{MR0379156} has shown that
there are no uniform solutions to Post's problem. This result predated
Theorem~\ref{thm:Steel}, which generalized it.
Downey and Shore \cite{MR1401736} later put further restrictions on any
possible solution. By using Theorem~\ref{thm:Steel}, they showed that any degree
invariant solution to Post's problem must be $\text{low}_2$ or
$\text{high}_2$. On the positive side, Slaman and Steel (unpublished) have
proved that there is a degree invariant solution to Post's problem
restricted to the domain of $\Sigma^0_3$ sets. Finally, Lewis has
constructed a degree invariant solution to Post's problem on a set of
sufficiently generic degrees \cite{Lewis05}.

Martin's conjecture has also inspired a couple theorems for functions from
$\cantor$ to $P(\cantor)$. Steel \cite{MR654792} has proved the existence of a natural
prewellorder on inner model operators using the uniform case of Martin's
conjecture. Slaman \cite{MR2143887} has proved an analogue of Martin's
conjecture for all Borel functions from $\cantor$ to $P(\cantor)$
satisfying certain natural closure conditions. This proof uses a technique
that is reminiscent of Theorem~\ref{thm:inc_order_preserving}, and relies on
a sharpening of the generalized Posner-Robinson theorem due to Shore and
Slaman~\cite{MR1739227}.

The metamathematics of Martin's conjecture has been the source of some
interesting results. Chong and Yu \cite{MR2500887} have constructed
uniformly Turing invariant $\Pi^1_1$ counterexamples to Martin's conjecture
when the hypothesis of $\ZF + \DC + \AD$ is replaced with $\ZFC + \V = \gL$.
They raised the question of the consistency strength of
Theorem~\ref{thm:Steel}. Chong, Wang, and Yu \cite{MR2576277} have proved
that the restriction of Theorem~\ref{thm:Steel} to $\PPi^1_{2n+1}$
functions is equivalent to $\SSigma^1_{2n+2}$ determinacy for all $n \geq
0$.

Next, we will turn to applications of the above in the field of countable
Borel equivalence relations. In this context, we will only need the
restriction of Martin's conjecture to Borel functions. In what follows,
when we assume Martin's conjecture, we will really mean that we assume its
consequences for Borel functions. 
The following characterization of Borel Martin's conjecture is an easy
consequence of
Theorems~\ref{thm:Slaman-Steel} and \ref{thm:Steel}.
\begin{thm}[Slaman and Steel]\label{thm:BMC}
  The following are equivalent:
  \begin{enumerate}
    \item Martin's conjecture restricted to Borel functions.
    \item If $f$ is Borel and Turing invariant, then either $f$ is constant a.e., or
    there is an ordinal $\alpha < \omega_1$ so that $f(x) \equiv_T
    x^{(\alpha)}$ a.e.
    \item If $f$ is Borel and Turing invariant, then there exists a uniformly Turing
    invariant $g$ such that $f \equiv_m g$. 
  \end{enumerate}
\end{thm}

Assuming Martin's conjecture, there is a particular fact about Turing
invariant functions that we will use several times. Given any subset $A$ of
$\cantor$, the \define{$\equiv_T$-saturation} of $A$ is defined to be the
smallest Turing invariant set containing it. Let $f: \cantor \to \cantor$
be a countable-to-one function that is Turing-invariant. Then the
$\equiv_T$-saturation of $\ran(f)$ must contain a Turing cone. This is
because $f$ cannot be constant a.e.\ and so it must be that $f(x) \geq_T x$
a.e. Hence, the complement of the $\equiv_T$-saturation of $\ran(f)$ cannot
contain a Turing cone. 

\subsection{Countable Borel equivalence relations}
\label{sec:BEQR}

Turing equivalence is an example of a countable Borel equivalence relation.
A \define{Borel equivalence relation} is an equivalence relation $E$ on a
Polish space $X$ that is Borel as a subset of $X \times X$. A Borel
equivalence relation is said to be \define{countable} if all its
equivalence classes are countable.

Suppose $E$ and $F$ are Borel equivalence relations on the Polish spaces $X$ and $Y$
respectively, and $f: X \to Y$ is a function. $f$ is said to be a
\define{homomorphism} from $E$ to $F$ if for all $x, y \in X$, we have $x E y
\rightarrow f(x) F f(y)$. In this language, a Turing invariant function is a
homomorphism from $\equiv_T$ to $\equiv_T$. $f$ is said to be a
\define{cohomomorphism} from $E$ to $F$ if for all $x, y \in X$, we have $f(x) F
f(y) \rightarrow x E y$. $f$ is said to be a \define{reduction} from $E$ to $F$ if $f$
is simultaneously a homomorphism and a cohomomorphism from $E$ to $F$. A
reduction from $E$ to $F$ induces an injection from the quotient $X/E$ to
the quotient $Y/F$. If the reduction $f$ is itself injective, then $f$ is said to be an
\define{embedding} of $E$ into $F$.

$E$ is said to be \define{Borel reducible} to $F$, noted $E
\leq_B F$, if there exists a Borel reduction from $E$ to $F$.
The class of countable Borel equivalence relations under 
$\leq_B$ has a rich structure that has been studied extensively. See for
instance the paper of Jackson, Kechris, and Louveau \cite{MR1900547}. 

We will be particularly interested in the phenomenon of universality. A
countable Borel equivalence relation $E$ is said to be \define{universal}
if for all countable Borel equivalence relations $F$, we have $F \leq_B
E$. It is known that there exist universal countable Borel equivalence
relations \cite{MR1149121}.

Martin's conjecture and the partial results surrounding it have turned out
to have many connections with the field of countable Borel equivalence
relations. Indeed, while Martin's conjecture says something very beautiful
and fundamental about Turing reducibility and the hierarchy of
definability, it is not so surprising that Martin's conjecture has been
more applicable in this area, rather than in recursion theory. Martin's
conjecture gives a complete classification of all homomorphisms from Turing
equivalence to itself, and homomorphisms are a basic object of study in the
area of countable Borel equivalence relations. 

A countable Borel equivalence relation $E$ is said to be \define{hyperfinite}
if $E = \bigunion_{i \in \omega} E_i$ where $E_0 \subseteq E_1 \subseteq
\ldots$ is an increasing sequence of Borel equivalence relations with
finite classes. Slaman and Steel \cite{MR960895} realized that if
$\equiv_T$ was hyperfinite, this would allow one to construct
counterexamples to Martin's conjecture. They showed that $\equiv_T$ is not
hyperfinite and they proved that a Borel equivalence relation is
hyperfinite if and only if it is induced by a Borel $\Z$ action. They obtained
these results independently from the work that was beginning on the field
of Borel equivalence relations at the time. Their last result is due
independently to Weiss \cite{MR737417}.

In \cite{MR960895}, Slaman and Steel posed further structurability
questions about Turing equivalence. These were answered by
Kechris \cite{MR1131739} via methods associated with the concept of
amenability. Amenability has since played a large role in the study of
Borel equivalence relations. 

Kechris \cite{MR1233813} has asked whether Turing equivalence is a
universal countable Borel equivalence relation. An affirmative answer to
this question would contradict Martin's conjecture, while Martin's
conjecture implies that $\equiv_T$ is not universal. See \cite{MR1770736}
for a more thorough discussion of this topic. Essentially, if there is a
reduction from $\equiv_T \disjointunion \equiv_T$ to $\equiv_T$, then the
range of the reduction on one of the copies of $\equiv_T$ must be disjoint
from a cone. 

A related question is due to Hjorth:
\begin{question}[Hjorth \cite{MR1815088}, \cite{MR1900547}]\label{q:Hjorth}
  If $E$ and $F$ are countable Borel equivalence relations on the
  Polish space $X$, and $E$ is universal and $E \subseteq F$, then must $F$
  be universal? 
\end{question}
If this question has an affirmative answer, then $\equiv_T$ is universal;
it is easy to find subsets of $\equiv_T$ that are universal. Of course, an
affirmative answer to this question would have many more implications. 

Let $E$ be a countable Borel equivalence relation on the Polish space $X$, and $\mu$
be a probability measure on $X$. Given a countable Borel equivalence
relation $F$, we say that $E$ is \define{$F$-ergodic} with respect to $\mu$
if every Borel homomorphism from $E$ to $F$ maps a set of measure $1$ into
a single $F$-class. $E$ is said to be simply \define{ergodic} with respect
to $\mu$ if it is $\Delta(Y)$-ergodic with respect to $\mu$ for every
Polish space $Y$, where $\Delta(Y)$ is the equivalence relation of equality
on $Y$. Define a subset $A$ of $X$ to be \define{$E$-invariant} if 
for all $x,y \in X$, if $x \in A$, and $x E y$, then $y \in A$.
Equivalently, $E$ is ergodic with respect to $\mu$ if and only if every
Borel $E$-invariant set has measure $0$ or $1$.

For the above definitions of ergodicity to make sense, $\mu$ can be a
measure on merely the $\sigma$-algebra of $E$-invariant Borel subsets of $X$,
as Martin measure is for $\equiv_T$. For example, $\equiv_T$ is ergodic
with respect to Martin measure.

Strong ergodicity results for $\equiv_T$ and Martin measure would be very
interesting. Let $E_0$ be the equivalence relation of equality mod finite
on $\cantor$. Thomas \cite{MR2563815} has shown that if $\equiv_T$ is
$E_0$-ergodic with respect to Martin measure, then $\equiv_T$ is not Borel
bounded. Borel boundedness is closely connected to the long open increasing
union problem for hyperfinite equivalence relations \cite{MR2322367}. It is
currently open whether there are any Borel equivalence relations that are
not Borel bounded.

Thomas has shown that Martin's conjecture implies that $\equiv_T$ is
$E_0$-ergodic with respect to Martin measure, and in fact, Martin's
conjecture implies the strongest ergodicity for $\equiv_T$ that is possible
with respect to Martin measure. 
If $E$ and $F$ are countable Borel equivalence
relations, then a \define{weak Borel reduction} is a countable-to-one Borel
homomorphism from $E$ to $F$. If there exists a weak Borel reduction from
$E$ to $F$, then we say that $E$ is \define{weakly Borel reducible} to $F$,
and write $E \leq^w_B F$. A countable Borel equivalence relation $E$ is
said to be \define{weakly universal} if $F \leq^w_B E$ for all countable
Borel equivalence relations $F$. Clearly, if $E$ is weakly universal, then
$\equiv_T$ is not $E$-ergodic with respect to Martin measure, since there
is a countable-to-one Borel homomorphism from $\equiv_T$ to $E$. Assuming
Martin's conjecture, Thomas has
proved the remarkable fact that the converse is true:

\begin{thm}[Thomas \cite{MR2563815}]\label{thm:Thomas_strong_ergodicity}
  Assume Martin's conjecture is true. Let $E$ be any countable Borel
  equivalence relation. Then exactly one of the following hold:
  \begin{enumerate}
    \item $E$ is weakly universal.
    \item $\equiv_T$ is $E$-ergodic, with respect to Martin measure.
  \end{enumerate}
\end{thm}
The proof of this theorem uses the fact that Martin's conjecture
implies that the saturation of the range of a countable-to-one Turing
invariant function must contain a Turing cone. We will discuss
Theorem~\ref{thm:Thomas_strong_ergodicity} more in
Section~\ref{sec:cor_and_open}.

Martin's conjecture appears to be closely connected to the
structure of the weakly universal countable Borel equivalence relations.
Thomas~\cite{MR2563815} has shown that assuming
Martin's conjecture, there are continuum many pairwise
$\leq_B$-incomparable weakly universal countable Borel equivalence
relations. These equivalence relations are products of the form $\equiv_T
\times E_\alpha$ where $\{E_\alpha : \alpha \in \cantor\}$ is a family of
non weakly universal countable Borel equivalence relations on $\cantor$ so that if
$\alpha \neq \beta$, then $E_\alpha$ is $E_\beta$-ergodic with respect to
Lebesgue measure. The proof of this result uses Popa's cocycle
superrigidity theorem to establish the existence of such a family of
$E_\alpha$, and then applies Theorem~\ref{thm:Thomas_strong_ergodicity} to
prove the $\leq_B$-incomparability of the product equivalence relations.

Thomas has also used Martin's conjecture to investigate weak universality
in another context. Say that a countable group $G$ is
\define{weakly action universal} if there is a Borel action of $G$ on a
Polish space $X$ so that the induced orbit equivalence relation $E^X_G$ is
weakly universal. Assuming Martin's conjecture,
Thomas has shown that $G$ is weakly action
universal if and only if the conjugacy relation on the subgroups of $G$ is
weakly universal \cite{ThomasUniversal}.

Aside from these results, the structure of the weakly universal
countable Borel equivalence relations is poorly understood. For example, it is open
whether there exists more than one weakly universal countable Borel
equivalence relation up to Borel reducibility; this is equivalent to
Hjorth's Question~\ref{q:Hjorth} having a negative answer. 

An appeal of Martin's conjecture is that it provides a dimension of
analysis that is completely orthogonal
to measure theory. This is particularly interesting
because all other known techniques for analyzing non-hyperfinite
countable Borel equivalence relations are measure-theoretic. 
Assuming Martin's conjecture, Thomas~\cite{MR2563815} has proved that
the complexity of a weakly universal countable Borel
equivalence relation always concentrates on a nullset. This is strong
evidence that techniques that are not purely measure-theoretic are needed to unravel the structure of the
weakly universal countable Borel equivalence relations.
\begin{thm}[Thomas~\cite{MR2563815}]\label{thm:Thomas_no_strongly}
Assume Martin's conjecture. If $E$ is any weakly
universal countable Borel equivalence relation on $X$, and $\mu$ is a Borel
probability measure on $X$, then there is a Borel set $B \subseteq X$ with
$\mu(B) = 1$ so
that $E \restrict B$ is not weakly universal. 
\end{thm}
Thomas has applied this theorem to show that assuming Martin's conjecture,
there does not exist a strongly universal countable Borel equivalence
relation. We will discuss this result more in
Section~\ref{sec:cor_and_open}.

The partial results on Martin's conjecture have also found
applications in the field of countable Borel equivalence relations.
Thomas~\cite{MR2563816} has used Theorem~\ref{thm:continous_MC} to show the
nonexistence of continuous Borel reductions between several equivalence
relations. These results are significant because in practice, most Borel
reductions are continuous. Montalb\'an, Reimann, and Slaman (unpublished)
have used Theorem~\ref{thm:Slaman-Steel} to show that $\equiv_T$ is not a
uniformly universal countable Borel equivalence relations. We will discuss
their theorem more in Section~\ref{sec:cor_and_open}.

\section{The universality of arithmetic equivalence}

It is natural to consider analogues of Martin's conjecture for other notions
of degree. In particular, this makes sense for any degree notion that has a
natural jump operator, and an analogue of Martin's measure (which generally
exists by a proof similar to Martin's proof in \cite{MR0227022}).
For instance, we can replace Turing reducibility by arithmetic
reducibility, the Turing jump with $x \mapsto
x^{(\omega)}$, and Martin measure with the arithmetic cone measure.

In \cite{MR960895}, Slaman and Steel note that their arguments for proving
Theorems~\ref{thm:Slaman-Steel} and \ref{thm:Steel} adapt to the
$\Delta^1_n$ degrees, and to the degrees of construtibility. However, they
also note that their proofs do not work for arithmetic equivalence. 

In later work, Slaman and Steel further investigated the
analogue of Martin's conjecture for arithmetic equivalence. They showed
that it is false, and the technique that they developed to build
counterexamples adapted to prove a stronger result: that arithmetic
equivalence is a universal countable Borel equivalence relation. It is this
long unpublished result that we will give a proof of in this section.
In Section~\ref{sec:cor_and_open}, we will use this result to prove some
theorems about universal countable Borel equivalence relations in general. 

\subsection{Basic definitions}

The Cantor space, noted $\cantor$, is the set of all functions from
$\omega$ to $2$. There is a bijection between subsets of $\omega$ and elements of
$\cantor$; a subset of $\omega$ is represented by its characteristic
function. In what follows, we freely move between these two
representations. 
Say that $x \in \cantor$ has a limit if there exists an $i$ such that
$\forall j > i$, $x(j) = x(i)$. Viewing $x$ has a subset of $\omega$, this
is the same as saying that $x$ is finite or cofinite.

Given two reals $x,y \in \cantor$, the \define{join} of $x$ and $y$ is the
real $x \join y$ defined by $(x \join y)(2n) = x(n)$ and $(x \join y)(2n+1)
= y(n)$ for all $n \in \omega$. The join of finitely many reals is defined
analogously. Fix a recursive bijection $\langle \cdot, \cdot\rangle $ from $\omega
\times \omega \to \omega$. Let $x$ be a subset of $\omega$. The
\define{$n$th column of $x$}, noted $x^{[n]}$, is the subset of $\omega$
defined by $x^{[n]} = \{m : \langle n,m\rangle  \in x\}$. Conversely, if $\{x_i\}_{i
\in \omega}$ are all reals, then the \define{join} of these reals, noted
$\bigoplus_{i \in \omega} x_i$, is the real whose $n$th column is $x_n$. 

Let $2^{< \omega}$ be the set of finite binary sequences. 
If $\sigma \in 2^{< \omega}$, then the length of
$\sigma$, noted $|\sigma|$, is the domain of $\sigma$.
If $x$ and $y$ are functions from $\omega$ to $2^{< \omega}$, define the
join of $x$ and $y$ similarly to the above.

If $x, y \in \cantor$, then $x$ is said to be \define{arithmetically
reducible} to $y$, noted $x \leq_A y$, if there is an $n$ so that $x$ has a $\Sigma^0_n$
definition relative to $y$. Equivalently, $y \geq_A x$ if there is an $n$ so that $y^{(n)}
\geq_T x$, where $y^{(n)}$ is the $n$th iterate of the Turing jump
relative to $y$. The associated countable Borel equivalence relation is
called arithmetic equivalence and is noted $\equiv_A$.

\subsection{The proof}
\label{sec:universality}

Let $F_2$ be the free group on two generators. We define the countable
Borel equivalence relation $\Einfty$ on $2^{F_2}$ in the following way: for
all $x, y \in 2^{F_2}$, let $x \Einfty y$ if and only if there exists a $g
\in F_2$ so that $x(h) = y(gh)$ for all $h \in F_2$. By a theorem of
Dougherty, Jackson, and Kechris \cite{MR1149121}, this is a universal
countable Borel equivalence relation. In order to show that arithmetic
equivalence is universal, we shall construct a Borel embedding $f: 2^{F_2}
\to \cantor$ of $\Einfty$ into $\equiv_A$. The particular properties of
$\Einfty$ will be unimportant to the proof which would work equally well with any
equivalence relation generated by the Borel action of a finitely generated
group.

When constructing $f$, we must satisfy two conflicting
requirements: we must make $f$ both a homomorphism and a cohomomorphism. 
In making $f$ a homomorphism, we must ensure that if $x \Einfty y$, then
$f(x) \equiv_A f(y)$. Let $\{w_i\}_{i \in \omega}$ be a recursive listing
of all the words in $F_2$. We will ensure that $f$ is a homomorphism by fixing a
way of ``coding'' $f(w_i \cdot x)$ into $f(x)$, for every $i$.

An obvious method of coding would be as follows. Let $g:
2^{F_2} \to \cantor$ be any Borel function. Then define $\tilde{g}: 2^{F_2}
\to \cantor$ to be \[ \tilde{g}(x) = \bigjoin_{i \in \omega} g(w_i \cdot
x).\] Given any $g$, we see that $\tilde{g}$ is a homomorphism from $\Einfty$
to $\equiv_A$; from $\tilde{g}(x)$, we can obtain any $\tilde{g}(w_i \cdot x)$ by
recursively permuting columns. The task, then, would be to construct a
Borel $g$ so that the associated $\tilde{g}$ was also a cohomomorphism.

Unfortunately, this approach is doomed to failure. If such a $\tilde{g}$
was a Borel reduction of $\Einfty$ to $\equiv_A$, it would also be a Borel
reduction of $\Einfty$ to $\equiv_T$, and it would be a uniform reduction.
Montalb\'an, Reimann, and Slaman have shown this is impossible. We will discuss
their result more in Section~\ref{sec:cor_and_open}.

Essentially, the problem is that the above coding is too easy to unravel
compared to how powerful arithmetic reductions are (or even how powerful
Turing reductions are). The coding we use must evidently match the power of
arithmetic equivalence more closely.

Our failed attempt above is interesting in the context of Hjorth's
Question~\ref{q:Hjorth}. Suppose $E$ and $F$ are countable Borel
equivalence relations, and $E \subseteq F$. A plausible intuition as to why
the universality of $E$ would imply the universality of $F$ is as follows:
it might be that any coding mechanism we could use to prove $E$ universal
must also work to prove $F$ universal, simply by taking a more ``generic''
function that uses this coding. However, the above example shows
that this is false; there are equivalence relations that are subsets of
$\equiv_A$ for which the above coding mechanism can be used to prove
universality.

The crux of the proof that $\equiv_A$ is universal is a method of coding so
that for every $n$, there are only finitely many words $w_i$ so that $f(w_i
\cdot x)$ is $\Sigma_n$ definable from $f(x)$. Hence, from the perspective
of a $\Sigma_n$ reduction, $f(x)$ behaves as essentially a finite join. By
taking a generic function $f$ for this type of coding, we can control these
finite joins, and ensure that our $f$ is a cohomomorphism. 

\begin{defn}
Given $y,z \in \cantor$, say that \define{$z$ jump codes $y$} if for
every $n$, $z^{[n]}$ has a limit, and $y(n) = \lim_{m} z^{[n]}(m)$. The
\define{Skolem function} for this jump coding is the function from $\omega$
to $\omega$ that maps $n$ to the least $i$ such that $\forall j \geq i
[z(\langle n,j\rangle ) = z(\langle n,i\rangle )]$. 
\end{defn}

The name of this coding derives from the fact that if $z$ jump codes $y$,
then $z' \geq_T y$. Indeed, using $z'$ as an oracle, we can compute both
$y$ and the Skolem function for this jump coding. Given $n$, find the least
$i$ so that $\forall j > i [z(\langle n,j\rangle ) = z(\langle n,i\rangle )]$, using the oracle
$z'$. Then the $n$th bit of $y$ is $z(\langle n,i\rangle )$. 

\begin{defn}
Let $x: \omega \to 2^{< \omega}$ be any function. For any real $y \in
2^\omega$, define $J(x,y) \in \cantor$ to be the real that jump codes $y$ via $x$.
Precisely, we mean that the $n$th column of $J(x,y)$ will be
\[\left(J(x,y) \right)^{[n]} = \begin{cases} 
x(n) 1 0000\ldots & \text{ if y(n) = 0}\\
x(n) 0 1111\ldots & \text{ if y(n) = 1}
\end{cases}
\]
\end{defn}
Hence, $J(x,y)$ jump codes $y$, and the Skolem function for the jump coding
is $n \mapsto |x(n)| + 1$, where $|x(n)|$ is the length of the finite
sequence $x(n)$. 

If $p$ is a partial function from $\omega$ to $2^{<
\omega}$ and $r$ is a partial function from $\omega$ to $2$ with $\dom(p)
\subseteq \dom(r)$, analogously define $J(p,r)$, a partial function from
$\omega$ to $2$, where the $n$th column of $J(p,r)$ is undefined if
$n \notin \dom(p)$. 

The idea of coding a real as a limit of columns has a long history in
recursion theory. The proof we will present uses jump codings of ``depth''
$\omega$. In this way, it is reminiscent of some constructions that have
been used to investigate the structure $\langle \mathbf{D},\leq_T,\prime\rangle $ of of
the Turing degrees under $\leq_T$ and the jump operator. See the papers of
Hinman and Slaman \cite{MR1133085}, and
Montalb\'an \cite{MR2000490}. 
 
In what follows, we will be using ideas from forcing in arithmetic, and in
particular, reals with limited Cohen genericity. Let $\Pw$ be the partial
order of finite partial functions from $\omega$ to $2^{< \omega}$ ordered
under inclusion. Say that a function $x$ from $\omega$ to $2^{< \omega}$ is
\define{arithmetically generic} if it meets every arithmetically definable
dense subset of $\Pw$. Similarly, finitely many functions $x_1,
\ldots x_n$ from $\omega$ to $2^{< \omega}$ are \define{mutually
arithmetically generic} if $(x_1, \ldots, x_n)$ meets every arithmetically
definable dense subset of $(\Pw)^n$.

We begin with a lemma whose proof is standard for the subject: 
\begin{lemma}
  If $x$, $z$, and $w$ are mutually arithmetically generic functions from $\omega$ to $2^{< \omega}$, then for all $n \in \omega$ and $y
  \in \cantor$, 
  \begin{enumerate}
    \item $\left(0^{(n)} \join J(x,y) \join z \right)' \equiv_T 0^{(n+1)}
    \join x \join y \join z$
    \item $0^{(n)} \join J(x,y) \join z \ngeq_T w$
  \end{enumerate}
\end{lemma}
\begin{proof}
  We prove part 1. Let $\Pt$ be the partial order of finite partial
  functions from $\omega$ to $2$ ordered under inclusion.

  Fix an $e$. Consider the set $D$ of pairs $(p,q) \in \left(\Pw\right)^2$
  such that for every $r \in \Pt$ with $\dom(p) =
  \dom(r)$, either $\varphi_e(0^{(n)} \join J(p,r) \join q)$ halts, or for
  every extension of $(p, q, r)$ to $(p^*, q^*, r^*)$, we have that $\varphi_e(0^{(n)}
  \join J(p^*,r^*) \join q^*)$ does not halt. We claim that $D$ is 
  is dense in $\left(\Pw\right)^2$. 

  Suppose $(p,q) \in \left(\Pw\right)^2$. We
  wish to show that that $(p,q)$ can be extended to meet $D$. Let
  $r_1, \ldots, r_n$ be a list of all elements of $\Pt$ such that $\dom(p) =
  \dom(r_i)$. Let $s_0 = \emptyset$, and $q_0 = q$. We will define an
  increasing sequence $s_1 \subseteq \ldots \subseteq s_{n}$ of elements of
  $\Pt$ and an increasing sequence $q_1 \subseteq \ldots \subseteq q_{n}$
  of elements of $\Pw$.
  
  Inductively, for $1 \leq i \leq n$, consider $0^{(n)} \oplus J(p,r_i)
  \disjointunion s_{i-1} \oplus q_{i-1}$, a partial function from $\omega$
  to $2$. Either there no extension of this partial function that makes
  $\varphi_e$ halt relative to it, or there is a finite such extension. If
  there is such an extension, let it be $0^{(n)} \oplus J(p,r_i)
  \disjointunion s_i \oplus q_i$, where $q_i$ extends $q_{i-1}$, where
  $s_i$ extends $s_{i-1}$, and the domain of $s_i$ is disjoint from
  $J(p,r_i)$. If there is no such extension, let $q_i = q_{i-1}$, and $s_i
  = s_{i-1}$. 

  Extend $p$ to any $\hat{p}$ so that for every $\langle j,k\rangle  \in
  \dom(s_{n})$, we have $\hat{p}(j)(k) = s_n(\langle j,k\rangle )$. Note that this
  means that for any $r \in \Pt$ with $\dom(r) = \dom(\hat{p})$,
  $J(\hat{p},r)$ will be an extension of $J(p,r) \disjointunion s_n$. It is
  clear that $(\hat{p},q_n)$ meets $D$. 

  If $x$ and $z$ are arithmetically generic, 
  then for each $e$, $0^{(n+1)}
  \join x \join y \join z$ can compute a place where $(x,z)$ meets $D$.
  Hence, from $0^{(n+1)} \join x \join y \join z$
  we can compute the $\Sigma^0_1$ theory of $0^{(n)} \join J(x,y) \join z$,
  and thus $0^{(n+1)} \join x \join y \join z \geq_T \left(0^{(n)} \join J(x,y) \join z\right)'$. Obviously
  $\left(0^{(n)} \join J(x,y) \join z\right)' \geq_T 0^{(n+1)} \join
  x \join y \join z$.

  We now proceed to part 2, whose proof is similar to part 1. Fix an $e$.
  The dense set that $(x,z,w)$ must meet is the set of 
  triples $(p,q,r) \in (\Pw)^3$ such that for every
  $s \in \Pt$ with $\dom(p) = \dom(s)$, there exists a $k$ such that
  $\varphi_e(0^{(n)} \join J(p,s) \join q)(k)\halts \neq r(k)$, or for
  every extension of $(p,q,s)$ to $(p^*,q^*,s^*)$, we have that
  $\varphi_e(0^{(n)} \join J(p,s) \join q)(k)$ does not halt. We leave the
  rest of the proof to the reader. 
\end{proof}

Note that for all $x_0, \ldots, x_n$ and $y_0, \ldots, y_n$, we have that 
\[J(x_0,y_0) \join \ldots \join J(x_{n}, y_{n}) \equiv_T J(x_0 
\join \ldots \join x_{n}, y_0 \join \ldots \join y_{n})\]
and that if $x_0, \ldots, x_n$ and $z_0, \ldots, z_n$ are all mutually
arithmetically generic, then $x_0 \join \ldots \join x_n$ and $z_0 \join
\ldots \join x_n$ are mutually arithmetically generic. Therefore, we can
conclude a more general fact: 

\begin{lemma}\label{lemma:final_forcing}
  If $x_0, \ldots, x_i$, $z_0, \ldots, z_j$, and $w$ are mutually
  arithmetically generic functions from $\omega$ to $2^{< \omega}$, then
  for all $n \in \omega$ and $y_0, \ldots, y_i \in \cantor$
  \begin{enumerate}
  \item $\displaystyle{\left(0^{(n)} \join J(x_0, y_0) \join \ldots \join
  J(x_i, y_i) \join z_0 \join \ldots \join
  z_j\right)'}$
  \vspace{-.2em}
  \begin{flushright}
  $\displaystyle{\equiv_T 0^{(n+1)} \join x_0 \join \ldots \join
  x_i \join y_0 \join \ldots \join y_i \join z_0 \join \ldots
  z_j}.$\end{flushright}
  \item $\displaystyle{0^{(n)} \join J(x_0, y_0) \join \ldots \join
  J(x_i, y_i) \join z_0 \join \ldots \join
  z_j \ngeq_T w}$
  \end{enumerate}
\end{lemma}

We are ready to prove the universality of arithmetic equivalence. 
\begin{thm}[Slaman and Steel]\label{thm:universality_of_arith}
  $\equiv_A$ is a universal countable Borel equivalence relation.
\end{thm}

\begin{proof}

Let $F_2 = \langle a,b\rangle $. To prove this theorem, we will construct a Borel
embedding of $E(F_2, 2)$ into $\equiv_A$. Let $g: 2^{F_2} \to
\left(2^{< \omega}\right)^\omega$ be a Borel
function so that for every distinct $x_0, \ldots, x_n \in 2^{F_2}$, we have
that $g(x_0), \ldots, g(x_n)$ are all mutually arithmetically generic
functions from $\omega$ to $2^{< \omega}$.
The definition of the embedding $f: 2^{F_2} \to \cantor$ is as
follows:
\[f(x) = J \left(g(x), f(a \cdot x) \join f(a^{-1} \cdot x) \join f(b \cdot x) \join
f(b^{-1} \cdot x) \right).\]
Note that while our definition of $f$ is self-referential, it is not
circular, as one can see by repeatedly expanding the terms involving $f$ on
the right hand side, using the definition of $f$. 

First, $f$ is a homomorphism. Recall that for all $x$ and $y$, 
$J(x,y)' \geq_T y$. Hence, if $x = w \cdot y$ where $w$ is a
word of $F_2$ of length $n$, then $(f(x))^{(n)} \geq_T f(y)$. 

Thus, we simply need to show that $f$ is a cohomomorphism. That is, if $x$ and $y$
are not $E(F_2,2)$ equivalent, then $f(x)$ and $f(y)$ are not
arithmetically equivalent.
Let $\{w_i : |w_i| < n\}$ be all words in $F_2$ of length $< n$, and let
$\{w_i : |w_i| = n\}$ be all words in $F_2$ of length $n$. Then 
\[(f(x))^{(n)} \equiv_T 0^{(n)} \join \bigjoin_{\{w_i: |w_i| < n\}} g(w_i \cdot x) \join
\bigjoin_{\{w_i : |w_i| = n\}} f(w_i \cdot x)\]
as one can see by inductively using part 1 of Lemma~\ref{lemma:final_forcing}.
(Recall that by definition,
$f(z)$ is of the form $J(g(z),w)$ for some $w$, and $g(z)$ is part of our
set of mutual generics). 
Hence, if $x$ and $y$ are not $\Einfty$ equivalent, by
part 2 of Lemma~\ref{lemma:final_forcing}, we see $(f(x))^{(n)} \ngeq_T
g(y)$ for all $n$. Hence, $(f(x))^{(n-1)}\ngeq_T f(y)$, since $f(y)' \geq_T
g(y)$.

\end{proof}

The original proof of the existence of pathological arithmetically
invariant functions was a similar construction to produce an embedding of
$\equiv_A$ into itself. Note that the range of the embedding in
Theorem~\ref{thm:universality_of_arith} is disjoint from the arithmetic
cone $\{x : x \geq_A 0^{\omega}\}$. Hence, embedding $\equiv_A$ into itself
via this technique produces an injective arithmetically invariant function
whose range is disjoint from an arithmetic cone.

\section{Corollaries and open problems}
\label{sec:cor_and_open}

A measure analogous to Martin measure exists for arithmetic equivalence.
It is called the \define{arithmetic cone measure}. An arithmetic cone is a
set of the form $\{x : x \geq_A y\}$ for some $y$. An arithmetically
invariant set has measure $1$ with respect to the arithmetic cone measure
if it contains an arithmetic cone, otherwise it has measure $0$. Martin's proof in \cite{MR0227022}
still works when Turing reducibility is replaced by arithmetic
reducibility. Hence, this function is indeed a measure on the
$\sigma$-algebra of arithmetically invariant sets.

The proof that arithmetic equivalence is universal
relativizes. That is, for every $x$, arithmetic equivalence relative to $x$
is universal. Equivalently, arithmetic equivalence restricted to any
arithmetic cone is universal. Using this fact, we can obtain several
interesting corollaries about universal countable Borel equivalence
relations in general. The results in this section are due to the first
author.

Jackson, Kechris, and Louveau \cite{MR1900547} have asked the following
question: suppose $E$ is a universal countable Borel equivalence relation
on $X$, and $B$ is an $E$-invariant Borel subset of $X$. Is one of $E \restriction B$
or $E \restriction (X \setminus B)$ universal? The answer is yes, and we
prove a stronger fact, originally posed as a question by
Thomas~\cite[question 3.20]{MR2500091}.

\begin{thm}\label{thm:ergodicity}
  Suppose $X$ and $Y$ are Polish spaces, $E$ is a universal countable Borel
  equivalence relation on $X$, and $f:
  X \to Y$ is any Borel homomorphism from $E$ to $\Delta(Y)$, where
  $\Delta(Y)$ is the relation of equality on $Y$.
  Then there
  exists a $y \in Y$ so that the restriction of $E$ to $f^{-1}(y)$ is a
  universal countable Borel equivalence relation.
\end{thm}
\begin{proof}
  First, note that it is enough to prove this for arithmetic equivalence. 
  Let $E$ and $f$ be as in the statement of the theorem, and
  let $g: \cantor \to X$ be a Borel reduction from $\equiv_A$ to $E$. If
  arithmetic equivalence restricted to $(f \circ g)^{-1}(y)$ is universal,
  then $E$ restricted to $f^{-1}(y)$ is universal. 
  
  Now let $f$ be a homomorphism from $\equiv_A$ to $\Delta(Y)$. Since
  arithmetic equivalence is ergodic with respect to the arithmetic cone
  measure, there must be a $y \in Y$ so that $f^{-1}(y)$ contains an
  arithmetic cone. Arithmetic equivalence restricted to this set is thus
  universal.
\end{proof}

The use of Borel determinacy in our proof raises an interesting
metamathematical question: must any proof of this theorem use Borel
determinacy? For instance, one could ask whether
Theorem~\ref{thm:ergodicity} implies Borel determinacy over some simple base
theory. We ask a weaker question of whether
Theorem~\ref{thm:ergodicity} shares a metamathematical property of Borel
determinacy: 
\begin{question}
  Does a proof of Theorem~\ref{thm:ergodicity} require the existence of
  $\omega_1$ iterates of the powerset of $\omega$?
\end{question}

Let $E$ be a countable Borel equivalence relation, and suppose $f$ is a
Borel homomorphism from $\equiv_A$ to $E$. Then $f$ is also a Borel
homomorphism from $\equiv_T$ to $E$. If $B \subseteq \cantor$ contains a
Turing cone, then the $\equiv_A$-saturation of $B$ must contain an
arithmetic cone. It is therefore possible to use ergodicity results about
Turing equivalence and Martin measure to obtain ergodicity results about
arithmetic equivalence and the arithmetic cone measure. We will apply
Thomas' Theorem~\ref{thm:Thomas_strong_ergodicity} in this way to prove an
analogous sort of ergodicity result for all universal countable Borel equivalence
relations.
\begin{thm}\label{thm:universal_ergodicity}
Assume Martin's conjecture is true. Suppose $E$ is a universal countable
Borel equivalence relation, and $F$ is an arbitrary countable Borel
equivalence relation. Then exactly one of the following holds:
\begin{enumerate}
  \item $F$ is weakly universal.
  \item For every Borel homomorphism $f$ of $E$ into $F$, there is
  a single $F$-class whose preimage $B$ has the property that $E \restrict B$ is
  universal.
\end{enumerate}
\end{thm}
\begin{proof}
  As in the proof of Theorem~\ref{thm:ergodicity}, we only need to prove this when $E$ is $\equiv_A$. Let $f$ be a
  homomorphism from $\equiv_A$ to $F$. Then $f$ is also a homomorphism from
  $\equiv_T$ to $F$, and hence by Theorem~\ref{thm:Thomas_strong_ergodicity},
  either $F$ is weakly universal, or there is a single $F$-class whose
  preimage $B$ contains a Turing cone. In this latter case, since $f$ is
  also a homomorphism from $\equiv_A$ to $F$, the preimage of this single
  $F$-class contains the $\equiv_A$-saturation of this Turing cone which is
  an arithmetic cone. Hence, since $B$ contains an arithmetic cone,
  $\equiv_A \restrict B$ is universal.
\end{proof}

In \cite{MR2563815}, Thomas proved a variant of this theorem where the assumption
that $E$ is universal is changed to say $E$ is weakly universal, and option
2 is changed to
say that $E \restrict B$ is weakly universal.
Theorem~\ref{thm:universal_ergodicity} strengthens this fact; by a result
of Miller and Kechris \cite{MR2500091}, E is a weakly universal countable
Borel equivalence relation if and only if there exists an $F \subseteq E$
that is a universal countable Borel equivalence relation.

In the proof of Theorem~\ref{thm:universal_ergodicity}, we have used the
ergodicity of $\equiv_T$ that follows from Martin's conjecture. 
However, we only need the weaker ergodicity which Martin's conjecture implies for
arithmetic equivalence. We isolate this in the following conjecture. 
It may be that it is easier to prove ergodicity
results for arithmetic equivalence and the arithmetic cone measure than it
is for Turing equivalence and Martin measure. 
\begin{conj}\label{conj:arithmetic_ergodicity}
  Let $E$ be any countable Borel equivalence relation. Then exactly one of
  the following holds:
  \begin{enumerate}
  \item $E$ is weakly universal.
  \item $\equiv_A$ is $E$-ergodic, with respect to the arithmetic cone
  measure.
  \end{enumerate}
\end{conj}

A special case of the above conjecture is quite interesting.
Thomas~\cite{MR2563815} has raised the question of whether $\equiv_T$ is
$E_0$-ergodic with respect to Martin measure. It is weaker to ask whether
arithmetic equivalence is $E_0$-ergodic with respect to the arithmetic cone
measure, but this would have similarly nice consequences. For example, it would imply that $\equiv_A$ is
not Borel bounded, and also that option 2 in
Theorem~\ref{thm:universal_ergodicity} holds when $F$ is $E_0$ without the
assumption of Martin's conjecture.

For our next application, we will need to recall some facts about pointed
perfect sets. Recall that a perfect subset of a Polish space $X$
is a closed subset of $X$ with no isolated points. Every perfect subset of
$\cantor$ can be realized as the paths $[T]$ through some infinite perfect
tree $T$ in $2^{< \omega}$. A \define{pointed perfect tree} $T$ is a tree $T$ so
that for all $x \in [T]$, $x \geq_T T$. A \define{pointed perfect set} is the paths
$[T]$ through some pointed perfect tree $T$. Pointed perfect sets arise
naturally in determinacy arguments, and have many nice properties. 

Given a perfect subset $[T]$ of $\cantor$, it is clear that $\cantor$ is
homeomorphic to $[T]$ via a canonical homeomorphism that preserves the
ordering on $\cantor$. This homeomorphism will preserve the Turing degrees
above $T$ if $T$ is pointed. That is, let $f: \cantor \to [T]$ be this
canonical homeomorphism. If $T$ is a pointed perfect tree, then for all $x
\geq_T T$, we have $x \equiv_T f(x)$, since both $x$ and $f(x)$ can compute
a representation of $T$ and hence also a representation of $f$. The analogous facts also hold for
arithmetic equivalence. In particular, if $T$ is a pointed perfect tree,
and $f: \cantor \to [T]$ is the canonical homeomorphism from $\cantor$ to
$[T]$, then for all $x \geq_A T$, we will have $x \equiv_A f(x)$. Hence, the 
restriction of $\equiv_A$ to any pointed perfect set is still a universal
countable Borel equivalence relation.

We will use the following lemma, which illustrates a
useful feature of pointed perfect sets in determinacy arguments. The proof
of this lemma is a slight variation of Martin's cone theorem in \cite{MR0227022}. 
\begin{lemma}[Martin \cite{MR0227022}]\label{lemma:pps}
  Assume $\ZF + \DC + \AD$. Then given any function $\pi: \cantor \to
  \omega$, there exists a pointed perfect set on which $\pi$ is constant.
\end{lemma}
\begin{proof}
  Consider the game where I plays $e \in \omega$ followed by $x \in
  \cantor$, and II plays $y \in \cantor$, where the players alternate
  playing bits of these reals as usual. Let II lose unless $y \geq_T x$, and if
  the game is not decided by this condition, then I wins if and only if $x
  \geq_T y$, and $\pi(x) = e$. 
  
  Given any strategy
  $\tau$ for II, I can win by playing $\pi(\tau)$ followed by $\tau$. Hence, I
  wins this game. Let $\sigma$ be a winning strategy for I. Then our
  pointed perfect set is the set of I's winning plays against II playing
  $\{y : y \geq_T \sigma\}$.
\end{proof}

If $E$ and $F$ are countable Borel equivalence relations, then
we say that $E$ is \define{Borel embeddable} in $F$ and write $E \embeds_B
F$ if there exists a Borel embedding of $E$ into $F$. 
We can use the above lemma to derive the following fact about universality
for embeddings.
\begin{thm}\label{thm:universal_embeddings}
  Let $E$ be a universal countable Borel equivalence relation.
  Then given any countable Borel equivalence relation $F$, it must be that $F \embeds_B E$. That is, not only is $F
  \leq_B E$ (since $E$ is universal), we can always find an injective Borel reduction.
\end{thm}
\begin{proof}
  First recall that the reduction from $\Einfty$ to $\equiv_A$ in the proof
  of Theorem~\ref{thm:universality_of_arith} is actually a Borel
  embedding, and not merely a Borel reduction. Of course, this remains true
  when the proof is relativized to any pointed perfect set. Recall also
  that Dougherty, Jackson, and Kechris \cite{MR1149121} have shown that
  every countable Borel equivalence relation embeds into $\Einfty$. Hence,
  it will be enough to show that there is an embedding of $\equiv_A$
  restricted to some pointed perfect set into $E$. 

  Since $E$ is countable universal, there is a Borel reduction $f$ from
  $\equiv_A$ to $E$. Using Lusin-Novikov uniformization (18.10, 18.15
  in \cite{MR1321597}), split $\cantor$ into countably many
  Borel pieces $\{B_i\}_{i \in \omega}$ so that $f$ is injective on each $B_i$.
  One of these $B_i$ must contain a pointed perfect set by
  Lemma~\ref{lemma:pps}.
\end{proof}

This theorem is an interesting counterpoint to the following theorem of
Thomas:
\begin{thm}[Thomas \cite{MR1903855}]
  There exist countable Borel equivalence relations $E$ and $F$ such that
  the equivalence classes of both $E$ and $F$ are all infinite, and $E
  \leq_B F$, and $F \leq_B E$, but it is not the case that $E \sqsubseteq_B F$. 
\end{thm}
  
Lemma~\ref{lemma:pps} also gives an easy proof of the following:
\begin{thm}\label{thm:count_pieces}
Suppose $E$ is a universal countable Borel equivalence relation on a Polish
space $X$, and let $\{B_i\}_{i \in \omega}$ be a partition of $X$ into
countably many (not necessarily $E$-invariant) Borel pieces. Then there exists
some $i$ such that $E \restriction B_i$ is a universal countable Borel
equivalence relation.
\end{thm}
\begin{proof}
  As in Theorem~\ref{thm:ergodicity}, we only need to prove this for
  $\equiv_A$. Let $\{B_i\}_{i \in \omega}$ be a partition of $\cantor$ into
  countably many Borel pieces. By Lemma~\ref{lemma:pps} above, one of these
  pieces must contain a pointed perfect set, and the restriction of
  $\equiv_A$ to any pointed perfect set is countable universal.
\end{proof}
If the $B_i$ in the above theorem are all $E$-invariant, this theorem
follows from Theorem~\ref{thm:ergodicity}. From this, we could also
conclude the general case since for any Borel $B$, $E \restriction B$ is
universal if and only if $E \restriction [B]_E$ is universal, where $[B]_E$ is the
$E$-saturation of $B$.

Theorem~\ref{thm:count_pieces} associates two natural $\sigma$-ideals to
every countable Borel equivalence relation. 
\begin{defn}
Let $E$ be a countable Borel equivalence relation on the Polish space $X$.
Define the \define{non-universal ideal of $E$} to be the Borel subsets $B$
of $X$ on which $E \restrict B$ is not universal. Define the
\define{non-weakly-universal ideal of $E$} to be the Borel subsets $B$ of
$X$ on which $E \restrict B$ is not weakly universal. 
\end{defn}
We will discuss these $\sigma$-ideals more
in what follows. They seem to be important for developing the theory of
universal and weakly universal countable Borel equivalence relations.

Let $E$ be a countable Borel equivalence relation on the Polish space $X$
which is equipped with an invariant ergodic Borel probability measure $\mu$. Say
that $E$ is \define{strongly universal} if $E \restrict B$ is
universal for every Borel $B \subseteq X$ with $\mu(B) = 1$. In
\cite{MR2500091}, Thomas asked whether there exists a strongly universal
countable Borel equivalence relation. Thomas later settled this question
under the assumption of Martin's conjecture using
Theorem~\ref{thm:Thomas_no_strongly}.
We are able to prove a weaker theorem without the assumption of Martin's
conjecture.

\begin{thm}\label{thm:measure_and_universality}
  Let $E$ be a universal countable Borel equivalence relation on the space
  $X$, and let $\mu$ be a Borel probability measure on $X$. Then there is a measure $0$ subset
  $B$ of $X$ for which $E \restriction B$ is a universal countable Borel
  equivalence relation.
\end{thm}

First, recall the following theorem of Sacks (we give the
relativized version of the theorem): 
\begin{thm}[Sacks \cite{MR0186554}]
  If $\mu$ is a Borel probability measure on $\cantor$, then for all $x \in
  \cantor$ such that there exists a representation $y \in \cantor$ of $\mu$
  such that $x >_T y$, the cone $\{z : z \geq_T x\}$
  has $\mu$-measure $0$.
\end{thm}
Hence, for any Borel probability measure, sufficiently complicated
cones are always nullsets. The same theorem is also true when Turing
reducibility is replaced with arithmetic reducibility. 
One way to see this is to first replace our measure $\mu$ with
an $\equiv_A$-quasi-invariant measure $\nu$ that
dominates $\mu$. Then we can find a Turing cone with $\nu$ measure 0, and
the $\equiv_A$-saturation of this cone will be the arithmetic cone with the
same base. 

Sacks' theorem implies that Martin measure and likewise the
arithmetic cone measure cannot be extended to probability measures on all
the Borel sets of $\cantor$. Pointed perfect sets seem to be as close as
we can get to being able to measure arbitrary Borel sets using these
measures. 

\begin{proof}[Proof of Theorem~\ref{thm:measure_and_universality}:]
  Again, we need only prove this for arithmetic equivalence; given any
  other universal countable Borel equivalence relation $E$ on the space $X$
  with Borel probability measure $\mu$, by Theorem~\ref{thm:universal_embeddings} let $f$ be a
  Borel embedding from $\equiv_A$ to $E$. Presuming the range of $f$ has
  positive $\mu$-measure, let $\nu$ be the
  measure on $\cantor$ defined by $\nu(A) =
  \frac{1}{\mu\left(f\left(\cantor\right)\right)} \mu(f(A))$. If $\equiv_A
  \restrict B$ is universal and $\nu(B) = 0$, then $E \restrict f(B)$ is
  also universal, and $\mu\left(f(B)\right) = 0$.

  As we have shown above, given any Borel probability measure $\mu$ on $\equiv_A$,
  there is an arithmetic cone with measure $0$, and $\equiv_A$ restricted to
  this cone is universal.
\end{proof}

The extra leverage that Thomas gets by assuming Martin's conjecture
is that for all Borel $B$, $\equiv_T \restrict B$ is weakly universal if and only if
$B$ contains a pointed perfect set. Hence if $B$ is $\equiv_T$-invariant,
then $\equiv_T \restriction B$ is
weakly universal if and only if $\equiv_T \restriction (\cantor \setminus B)$ is not
weakly universal.

This exact classification of the non-weakly-universal ideal for $\equiv_T$
that follows from Martin's conjecture seems very useful.
\begin{question}
  Are there ``nice'' characterizations of the non-weakly-universal ideals
  of naturally occurring weakly universal countable Borel equivalence relations?
\end{question}
One could also ask the same question for universal countable Borel
equivalence relations, and the non-universal ideal.
Theorem~\ref{thm:measure_and_universality} seems to rule out
characterizations that are based purely on measure theory.
The fact that these ideals do not seem to be measure-theoretic is 
very interesting, since all known theorems in the
field of countable Borel equivalence relations that distinguish between
non-hyperfinite countable
Borel equivalence relations are based on measure theory.

Marks \cite{MarksPhD} raises a question that seems to be relevant. Define a
countable Borel equivalence relation $E$ to be \define{measure universal} if 
for every countable Borel equivalence relation $F$ on a Polish space $X$
equipped with a Borel probability measure $\mu$, there exists a $B \subseteq X$ that
is $F$-invariant, and $\mu(B) = 1$, so that $F \restrict B$ is Borel reducible
to $E$. 
\begin{question}[Marks \cite{MarksPhD}]
  If $F$ is a countable Borel equivalence relation that is measure
  universal, is $F$ universal?
\end{question}
This question was motivated by a result in Marks \cite{MarksPhD} that
many-one equivalence and recursive isomorphism are measure universal. It
remains open whether these equivalence relations are universal. 

\subsection{Some questions on uniformity}
\label{sec:conjectures}

Borel Martin's conjecture reduces purely to a question
about uniformity, as shown in Theorem~\ref{thm:BMC}.
If we embrace Martin's conjecture and ponder what larger principle it might
embody, we are naturally led to the possibility that similar principles of
uniformity might exist amongst a much wider class of equivalence relations,
even though the original form of Martin's conjecture appears to hinge on
specific properties of Turing equivalence that do not generalize to many
equivalence relations. 
This is an intriguing possibility
that would lead to a compelling theory providing a systematic way to
explain many phenomena. In this section,
we shall adopt such a viewpoint and pose several questions about uniformity
in broader contexts. Our questions will be phrased so that affirmative
answers would be the most natural from the above
perspective. However, even negative answers would be
very interesting as they might provide starting points for constructing
counterexamples to Martin's conjecture. 

Admittedly, we currently have little evidence supporting the viewpoint we
shall outline. Such questions of uniformity are presently poorly
understood, and these issues appear quite deep. We know of no general
theorems in this area, and not even any theorems in specific cases, beyond
the work in \cite{MR960895}. Likewise, there are few examples of
nonuniformity which seem to have much bearing on the questions we will ask.

Consider, for instance, the case of arithmetic equivalence. We have seen above
that the analogue of Martin's conjecture fails for arithmetic equivalence.
However, many questions about arithmetically invariant functions remain,
and chief among them is the arithmetic analogue of 
Conjecture~\ref{conj:SC}. Say that an arithmetically invariant function $f$
is \define{uniformly arithmetically invariant} if there exists
a function $u: \omega^2 \to \omega^2$ such that if $x \equiv_A y$ via
$(i,j)$, then $f(x) \equiv_A f(y)$ via $u(i,j)$.

\begin{question}[$\ZF + \DC + \AD$]\label{question:uniform_arith}
  If $f: \cantor \to \cantor$ is arithmetically invariant, then is there
  a uniformly arithmetically invariant $g$ which is defined on an arithmetic
  cone so that $f(x) \equiv_A g(x)$ on an arithmetic cone? 
\end{question}

Little is known about this question. The pathological arithmetically
invariant functions that were constructed in Section~\ref{sec:universality}
are all uniformly invariant, and the technique used to construct them gives
no hint about questions of nonuniformity. 
Likewise, the following conjecture is particularly intesting in light of
the earlier results in this section:

\begin{conj}\label{question:A_to_T}
  If $f$ is a Borel homomorphism from $\equiv_A$ to $\equiv_T$, then there
  exists a Borel homomorphism $g$ from $\equiv_A$ to $\equiv_T$ so that
  $f(x) \equiv_A g(x)$ on an arithmetic cone and $g$ is uniform in the
  sense that there exists a function $u: \omega^2 \to \omega^2$ so that
  for all $x$ and $y$, if $x \equiv_A y$ via $(i,j)$, then $g(x) \equiv_T
  g(y)$ via $u(i,j)$. 
\end{conj}

This conjecture is weaker than Borel Martin's conjecture, but it has
many of the same consequences for the theory of Borel equivalence
relations. For instance, it implies that
Conjecture~\ref{conj:arithmetic_ergodicity} is true, and hence that
Theorem~\ref{thm:universal_ergodicity} is true without the assumption of
Martin's conjecture. It also implies that Turing equivalence is not a
universal countable Borel equivalence relation. This is because it rules out an embedding of
$\equiv_A$ into $\equiv_T$; for all $\alpha$ with $\omega \leq \alpha <
\omega_1$, the map $x \mapsto x^{(\alpha)}$ is not a reduction of
arithmetic equivalence restricted to any arithmetic cone into $\equiv_T$.

It seems interesting to consider questions of uniformity analogous to
Conjecture~\ref{question:A_to_T} for most other equivalence relations from
recursion theory. Even more generally, we shall formulate such questions
for arbitrary weakly universal countable Borel equivalence relations.

In the field of countable Borel equivalence relations, uniformity is
usually discussed in the context where an equivalence relation $E$ is
equipped with a group $G$ and a Borel action of $G$ that
generates $E$. If $G$ is a countable group that acts on a Polish space
$X$, define the orbit equivalence relation $E^X_G$ where $x E^X_G
y$ if and only if there exists a $g \in G$ such that $g
\cdot x = y$. By a result of  Feldman and Moore \cite{MR0578656}, every
countable Borel equivalence relation $E$ is of this form.

In recursion theory, however, most equivalence relations (for instance
$\equiv_T$ and $\equiv_A$) are not naturally generated by group actions.
Further, while the Feldman-Moore theorem guarantees that it is possible to
construct group actions that generate recursion-theoretic equivalence
relations, it is often difficult to work with them using tools from
recursion theory. Thus, when working with weakly universal equivalence
relations, it seems natural to expand our definitions, since
recursion-theoric equivalence
relations seem so important in this context. 

\begin{defn} 
  A \define{generating family of partial Borel functions} on a Polish space
  $X$ is a countable set $\{\phi_i\}$ of partial Borel functions on $X$
  that is indexed by natural numbers (though we will omit the indexing for
  clarity), contains the identity function, and is closed under
  composition. Precisely, by partial Borel function on $X$, we mean that
  the domain of each $\phi_i$ is a Borel subset of $X$, and that the
  function $\phi_i: \dom(\phi_i) \to X$ is a Borel function. Such a
  generating family induces an equivalence relation $E^X_{\{\phi_i\}}$,
  where $x E^X_{\{\phi_i\}} y$ if and only if there exists $\phi_i$ and
  $\phi_j$ so that $\phi_i(x) = y$, and $\phi_j(y) = x$. In this case, say
  that $x E^X_{\{\phi_i\}} y$ via $(i,j)$.

  We will use the notation $E^X_{\{\phi_i\}}$ to indicate a countable Borel
  equivalence relation on the Polish space $X$ that is induced by the
  generating family of partial Borel functions $\{\phi_i\}$.
\end{defn}

For example, the Turing reductions are a generating family of partial
Borel functions for $\equiv_T$, and the arithmetic reductions are a generating
family of partial Borel functions for $\equiv_A$. 

\begin{defn}
  Suppose $E^X_{\{\phi_i\}}$ and $E^Y_{\{\psi_i\}}$ are countable Borel
  equivalence relations on the Polish spaces $X$ and $Y$ induced by the
  generating families of partial Borel functions $\{\phi_i\}$ and $\{\psi_i\}$.
  Say that a partial homomorphism $f: X \to Y$ is \define{uniform (with
  respect to $\{\phi_i\}$ and $\{\psi_i\} $)} if there exists a function
  $u: \omega^2 \to \omega^2$ such that for all $x, y \in X$, if $x
  E^X_{\{\phi_i\}} y$ via $(i,j)$, then $f(x) E^Y_{\{\psi_i\}} f(y)$ via
  $u(i,j)$.
\end{defn}

Let $E^X_G$ be a countable Borel equivalence relation generated by a
Borel action of the countable group $G$. To each $g \in G$
we associate the Borel automorphism $\phi_g(x) = g \cdot x$ of X, and
so $E^X_G$ is the same as the equivalence relation
$E^X_{\{\phi_g\}}$ given by the generating family
$\{\phi_g\}$. Say that $E^X_G$ is
\define{freely generated} if for all $g \in G$ and for all $x \in
X$, if $g \cdot x = x$ then $g = 1$. In the context where
$E^X_G$ and $E^Y_H$ are freely generated countable Borel
equivalence relations, then in the measure context, the uniformity defined above is equivalent 
to saying that the cocycle induced by a homomorphism
between equivalence relations is cohomologous to a group homomorphism.
Thus, the sort of uniformity implied by Martin's conjecture (and more
general questions which we will soon discuss) shares something of the same
spirit as cocycle supperrigidity which was introduced by
Zimmer \cite{MR776417}. Cocycle superrigidity has since played an important part
in developing the theory of countable Borel equivalence relations.

We will now formulate an analogue of Conjecture~\ref{question:A_to_T} for
homomorphisms between two arbitrary weakly universal countable Borel
equivalence relations. In posing such a question, we need an equivalent of
a set of Martin measure $1$; a ``large'' subset of the underlying space on
which we demand that uniformity is witnessed. Here, the natural candidate
is an invariant Borel set such that the restriction of the equivalence
relation to this set is still weakly universal. This notion agrees with
Martin measure in the case of $\equiv_T$. The fact that the
non-weakly-universal Borel sets form a $\sigma$-ideal is further evidence that
this is a good notion of ``largeness''.

Note that it is certainly not the case that for every homomorphism $f$
between equivalence relations $E^X_{\{\phi_i\}}$ and $E^Y_{\{\psi_i\}}$,
there exists a uniform homomorphism $g$ from $E^X_{\{\phi_i\}}$ to
$E^Y_{\{\psi_i\}}$ so that $f(x) E^Y_{\{\psi_i\}} g(x)$ for all $x$.
There are trivial counterexamples exploiting partiality of generating
families. Even in the case of equivalence relations freely
generated by the actions of countable groups, not every cocycle associated
to a homomorphism of such equivalence relations is cohomologous to a group
homomorphism. 

We proceed to our question:

\begin{question}\label{question:uniformity_of_homomorphisms}
  Let $E^X_{\{\phi_i\}}$ and $E^Y_{\{\psi_i\}}$ be weakly universal
  countable Borel equivalence relations. Now suppose $f$ is a homomorphism
  from $E^X_{\{\phi_i\}}$ to $E^Y_{\{\psi_i\}}$. Must there be an
  $E^X_{\{\phi_i\}}$-invariant subset $B$ of $X$ such that $E^X_{\{\phi_i\}}
  \restrict B$ is weakly universal, and a uniform partial homomorphism $g: X \restrict B \to Y$ from
  $E^X_{\{\phi_i\}}$ to $E^Y_{\{\psi_i\}}$ such that $f(x) E^Y_{\{\psi_i\}}
  g(x)$ for all $x \in B$? 
\end{question}

As stated, an affirmative answer to this question does not seem to
generalize Borel Martin's conjecture. For this to be the case, we also need
a characterization of the non-weakly-universal ideal of Turing equivalence.

\begin{defn}
  Suppose that $E^X_{\{\phi_i\}}$ is a weakly universal countable Borel
  equivalence relation.
  Say that $E^X_{\{\phi_i\}}$ is \define{uniformly weakly universal} if
  for every $E^Y_{\{\psi_i\}}$, there is a uniform weak Borel
  reduction of $E^Y_{\{\psi_i\}}$ to $E^X_{\{\phi_i\}}$.
\end{defn}
A related definition will also be of interest to us:
\begin{defn}
  Suppose that $E^X_{\{\phi_i\}}$ is a universal countable Borel
  equivalence relation.
  Say that $E^X_{\{\phi_i\}}$ is \define{uniformly universal}
  if for every
  $E^Y_{\{\psi_i\}}$, there is a uniform Borel reduction from
  $E^Y_{\{\psi_i\}}$ to $E^X_{\{\phi_i\}}$. 
\end{defn}

All known universal and weakly universal countable Borel equivalence
relations are uniformly universal and uniformly weakly universal when
equipped with natural generating families. For instance, $\Einfty$ is
uniformly universal with respect to the shift action of $F_2$, arithmetic
equivalence is uniformly universal with respect to the family of arithmetic
reductions, and Turing equivalence is uniformly weakly universal with
respect to the family of Turing reductions. Note that uniform universality
and uniform weak universality both depend on the generating family that we
use. In particular, every universal/weakly universal countable Borel
equivalence relation $E$ is uniformly universal/uniformly weakly universal
for some generating family, as one can see by applying
Theorem~\ref{thm:universal_embeddings}. 

\begin{question}\label{question:uniformity_of_weakly_universal}
  Suppose that $E^X_{\{\phi_i\}}$ is a weakly universal countable Borel
  equivalence relation.
  Must $E^X_{\{\phi_i\}}$ be weakly uniformly universal?
\end{question}

A positive answer to Question~\ref{question:uniformity_of_weakly_universal}
would provide the characterization
of the non-weakly-universal ideal of $\equiv_T$ that is needed so that a
positive answer to Question~\ref{question:uniformity_of_homomorphisms}
implies Borel Martin's conjecture. 

\begin{prop}\label{prop:questions_imply_MC}
  An affirmative answer to both
Question~\ref{question:uniformity_of_weakly_universal} and
Question~\ref{question:uniformity_of_homomorphisms} implies Borel Martin's
conjecture. 
\end{prop}
\begin{proof}
  In what follows, equip $\equiv_T$ with the generating family of Turing
  reductions. Consider a $\equiv_T$-invariant subset $B$ of $\cantor$ so
  that $\equiv_T \restrict B$ is weakly universal. Supposing
  Question~\ref{question:uniformity_of_weakly_universal} has a positive
  answer, there must be a uniform countable-to-one homomorphism from
  $\equiv_T$ to $\equiv_T \restriction B$. Hence, $B$ must
  contain a cone.

  Now let $f: \cantor \to \cantor$ be a homomorphism from $\equiv_T$ to
  $\equiv_T$. Supposing Question~\ref{question:uniformity_of_homomorphisms}
  has a positive answer, there is a weakly universal $\equiv_T$-invariant
  subset $B$ of $\equiv_T$ and a uniform homomorphism $g$ from $\equiv_T$
  to $\equiv_T$ so that $f(x) \equiv_T g(x)$ on $B$. By the above, $B$
  contains a Turing cone.
\end{proof}

Versions of Questions~\ref{question:uniformity_of_homomorphisms} and
\ref{question:uniformity_of_weakly_universal} that extend beyond the Borel
realm also seem interesting. For instance, one can consider these questions
in the context of $\ZF + \DC + \AD^+$, where we use $\AD^+$ as opposed to
$\AD$ so that there is enough uniformization so that weakly universal
countable equivalence relations exist. Here, positive answers to these questions
would imply the analogue of Martin's conjecture for the $\Delta^1_n$
degrees and the degrees of construtibility via versions of
Proposition~\ref{prop:questions_imply_MC} and
Theorems~\ref{thm:Slaman-Steel} and \ref{thm:Steel} for these degree
notions.

We finish by stating a question about uniform universality which is
related to Question~\ref{question:uniformity_of_weakly_universal}, and has
arisen naturally in other work.

If we generate $\equiv_T$ by the family of Turing reductions, then by
applying the uniform case of Martin's conjecture in
Theorem~\ref{thm:Slaman-Steel}, we see that there cannot be any uniform
embedding of $\equiv_T \disjointunion \equiv_T$ into $\equiv_T$. Hence,
$\equiv_T$ is not uniformly universal for the generating family of Turing
reductions. Montalb{\'a}n, Reimann, and Slaman (unpublished) have
generalized this fact to show that $\equiv_T$ is not uniformly universal
when it is equipped with a certain generating group. The central technical
obstacle in their work is adapting the game proofs of \cite{MR960895} to
work with this group. Montalb\'{a}n, Reimann, and Slaman posed the question
of what role uniformity plays in universality proofs, and more broadly in
the theory of countable Borel equivalence relations. Their question has
been inspiration for much of this section. In this spirit, we ask the
following:

\begin{question}\label{question:universal_families}
  Suppose that $E^X_{\{\phi_i\}}$ is a universal countable
  Borel equivalence relation. Must $E^X_{\{\phi_i\}}$ be uniformly
  universal?
\end{question}

Uniform universality in this more general setting has also shown up naturally in recent work of
Marks \cite{MarksPhD}, who has shown that many-one equivalence is uniformly
universal with respect to the generating family of many-one reductions if and only if
a question in Borel combinatorics has an affirmative answer. A resolution
of Question~\ref{question:universal_families} would help clarify this situation. 

\bibliography{arith_equiv_bib}

\end{document}